\documentclass{article}
\usepackage{graphicx,epstopdf}
\usepackage[centertags]{amsmath}
\usepackage{amsfonts,amssymb,amsthm,mathrsfs,newlfont,url,textcomp,geometry}
\usepackage[usenames, dvipsnames]{color}
\usepackage{soul,multicol,units,afterpage,mathtools,caption,multirow,array,booktabs,adjustbox,listings}
\usepackage{fourier}
\DeclareRobustCommand{\bigO}{%
  \text{\usefont{OMS}{cmsy}{m}{n}O}%
}

\newtheorem{theorem}{Theorem}
\newtheorem{lemma}{Lemma}
\newtheorem{proposition}{Proposition}

\title{Terminating Poincar{\'e} asymptotic expansion of the Hankel transform of entire exponential type functions}
\author{Nathalie Liezel R. Rojas and Eric A. Galapon \\
{\normalsize Theoretical Physics Group,  National Institute of Physics } \\
{\normalsize University of the Philippines Diliman } \\
{\normalsize Quezon City, 1101 Philippines}}
\date{\today}

\begin{document}
\maketitle

\begin{abstract}
We perform an asymptotic evaluation of the Hankel transform, $\int_0^{\infty}J_{\nu}(\lambda x) f(x)\mathrm{d}x$, for arbitrarily large $\lambda$ of an entire exponential type function, $f(x)$, of type $\tau$ by shifting the contour of integration in the complex plane. Under the situation that $J_{\nu}(\lambda x)f(x)$ has an odd parity with respect to $x$ and the condition that the asymptotic parameter $\lambda$ is greater than the type $\tau$, we obtain an exactly terminating Poincar{\'e} expansion without any trailing subdominant exponential terms. That is the Hankel transform evaluates exactly into a polynomial in inverse $\lambda$ as $\lambda$ approaches infinity. 
\end{abstract}

\hspace{1em} {\small Keywords: Poincar{\'e} asymptotic expansion, Hankel integral transform, exponential type function}

\section{Introduction}
The Hankel transform of order $\nu$ of a function $f(x)$ is given by
\begin{align}\label{eqn:HankelInt}
    I (\lambda) = \int_0^{\infty} f(x) \, J_{\nu} (\lambda x) \mbox{d} x,
\end{align}
where $J_{\nu} (z)$ is the Bessel function of the first kind. As $\lambda$ grows large, this class of integrals exhibit strong oscillations such that analytical and numerical methods become insufficient in solving them \cite{gabutti1985asymptotic}, which leads into considering their asymptotic estimates in order to obtain useful results. One representation of this integral is its Poincar{\'e} asymptotic expansion (PAE) which is given by
\cite{wong2001asymptotic,mackinnon1972asymptotic,zayed1982asymptotic,lopez2007asymptotic,lopez2010asymptotic}
\begin{equation}\label{eqn:poin}
\int_0^{\infty} f(x) J_{\nu}(\lambda x) \mbox{d}x \sim \frac{1}{2}\sum_{s=0}^{\infty} \frac{f^{(s)}(0)}{s!} \frac{\Gamma\left(\frac{1}{2}(\nu+s+1)\right)}{\Gamma\left(\frac{1}{2}(\nu-s+1)\right)}\left(\frac{2}{\lambda}\right)^{s+1}, \hspace{2mm} \lambda \rightarrow \infty ,
\end{equation}
provided that $f(x)$ is infinitely differentiable at the origin. This  expansion can be derived using the Mellin transform method \cite{wong2001asymptotic}. In certain situations, the PAE fails to provide a complete information when there are missed out exponentially subdominant terms in the expansion or it fails to be meaningful at all when the integral happens to be exponentially small for which case each term in the PAE vanishes. 
Recovering these exponentially small terms contribute significantly to the numerical accuracy and exactification of the PAE of the Hankel transform \cite{galapon2014exactification}.

Obtaining the subdominant exponentially small terms has been explored in few works. Frenzen and Wong \cite{frenzen1985note} performed an asymptotic estimate of the Hankel transform where the Bessel function $J_{\nu} (z)$ of order $\nu=0$, by shifting the contour of integration in the complex plane, and showed that the exponentially small terms can be obtained along the path that is parallel to the real line. Gabutti \cite{gabutti1985asymptotic} also considered obtaining an asymptotic estimate of the Hankel transform for a general case of nonzero order of $\nu$ by considering the kernel of the Hankel transform as Laguerre polynomials and showed results that have exponentially small behaviour. Further, Gabutti and Lepora \cite{gabutti1987novel} considered another case where the order of $\nu$ is a half-integer and used Laplace transforms for the function $f(x)$ in the Hankel transform. They were able to obtain asymptotic expansions of the Hankel transform which are exponentially small. Although the above mentioned works were only asymptotic estimates of the Hankel transform in which some of the results are not applicable to all cases of integer orders of $\nu$, it can be seen in the results that the exponentially small terms are present in the subdominant terms of the asymptotic expansion. But the numerical use of these asymptotic expansions without considering its error terms can end up having inaccurate results \cite{olver1997asymptotics}.

Further study of the Hankel transform was conducted by Galapon and Martinez \cite{galapon2014exactification} in which they solved the same form as in equation \eqref{eqn:poin} using the distribution theory approach. In their study, they derived the exponentially small terms that are subdominant to the PAE and investigated specific cases and examples depending on the order of $\nu$ which resulted to exactifying the value of the Hankel integral.
In one case involving half-integer orders of $\nu$, they showed that the exponentially small terms in the asymptotic expansion vanished and the obtained result becomes a terminating PAE which were able to provide an exact value of the considered integral.
For the case when $\nu=0$, an example of the Hankel integral was also solved. The obtained results show that the PAE becomes zero and the remaining exponentially small terms in the expansion yield the exact value of the integral. In another case for any positive integer $\nu$, they showed that the obtained result includes the PAE of the integral, which terminated, along with the exponentially small terms, and this result achieved the exact value of the integral. This result was obtained by considering the integrand of the Hankel transform be odd, such as when $f(x)$ is odd and $J_{\nu} (\lambda x)$ is even, or when $f(x)$ is even and $J_{\nu} (\lambda x)$ is odd.
From these presented results, we now reflect these findings in our work.
We aim to obtain a terminating Poincar{\'e} asymptotic expansion of the Hankel transform in equation \eqref{eqn:poin} for integer orders of $\nu$ where the exponentially small terms do not contribute.

In this paper, we consider the function $f(x)$ in the Hankel transform to be an entire exponential function of type $\tau$ and perform a shifting of the contour integration in the complex plane, following the method presented in \cite{frenzen1985note,galapon2016internal}. A function is said to be an entire function of exponential type $\tau$ if it has a complex extension $f(z)$ in the complex plane which satisfies the asymptotic condition that for sufficiently large values of $|z|$ and for any arbitrarily small $\epsilon>0$, the inequality $|f(z)| \leq   e^{(\tau + \epsilon) |z|}$ holds \cite{galapon2016internal,boas1954entire}.
We find that when $f(x)$ is an entire exponential type function and the integrand, $f(x) J_{\nu} (\lambda x)$ is odd, the resulting series expansion terminates without trailing exponentially small terms, provided that the condition $\lambda>\tau$ is satisfied. 
In particular, we find that the resulting expansion terminates to a polynomial in inverse order of $\lambda$ as $\lambda \to \infty$.

The rest of the paper is organized as follows. In Section 2, we present theorems for solving for the integral in equation \eqref{eqn:HankelInt} where $f(x)$ is an entire exponential type function, and show that the obtained result is a terminating PAE . In  Section 3, we apply these results to exponential type functions involving gamma functions and show that the results exactify the value of the Hankel transforms.

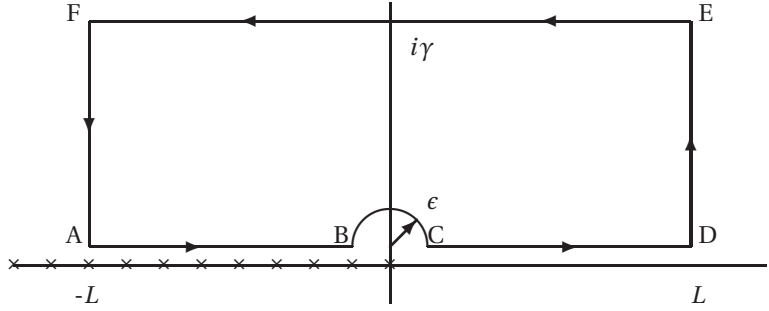
\begin{figure} 
\begin{center}
\setlength{\unitlength}{1cm}
\begin{picture}(12,4)
  \thicklines
  \put(1,0){\line(1,0){10}} 
  \put(6,-0.5){\line(0,1){4}} 
  \multiput(0.88,-0.07)(0.5,0){11}{$\times$} 
  \put(2,0.25){\vector(1,0){1.5}}
  \put(3.5,0.25){\line(1,0){2}}
  \put(6,0.25){\oval(1,1)[t]} 
  \put(6,0.25){\vector(1,1){0.36}}
  \put(6.5,0.75){$\epsilon$}
  \put(6.5,0.25){\vector(1,0){2}}
  \put(8.5,0.25){\line(1,0){1.5}}
  \put(10,0.25){\vector(0,1){1.5}}
  \put(10,1.75){\line(0,1){1.5}}
  \put(2,3.25){\vector(0,-1){1.5}}
  \put(2,1.75){\line(0,-1){1.5}}
  \put(10,3.25){\vector(-1,0){2}}
  \put(4,3.25){\line(-1,0){2}}
  \put(8,3.25){\vector(-1,0){4}}
  \put(1.8,-0.5){-$L$}
  \put(10,-0.5){$L$}
  \put(6.25,2.8){$i \gamma$}
  \put(1.7,0.3){A}
  \put(5.25,0.3){B}
  \put(6.5,0.3){C}
  \put(10.1,0.3){D}
  \put(10.1,3.25){E}
  \put(1.7,3.25){F}
\end{picture}
\end{center}
\caption{The path of integration.}
\label{fig:Figure1}
\end{figure}

\section{Hankel transform of entire exponential type function}
 
We now obtain the Hankel transform of entire exponential type function $f(x)$ in equation \eqref{eqn:HankelInt} where $J_{\nu} (\lambda x)$ is of integer order $\nu$ for a natural number $\nu$ and $\lambda>0$. The integral converges under the condition $f(x)\to 0$ as $|x|\to\infty$. Additionally, $|J_{\nu} (x)| \leq 1$ for $\nu \geq 0$ and the integral in \eqref{eqn:HankelInt} has the bound $|I| \leq \int_0^\infty |f(x)| \, \text{d} x <\infty$ when $f(x)$ is absolutely integrable in the interval $[0,\infty)$.
When the integrand is odd, we recognize two cases: when $f(x)$ is odd and $J_\nu (\lambda x)$ is even, and when $f(x)$ is even and $J_\nu (\lambda x)$ is odd.

\subsection{Case 1}
\begin{theorem}\label{theorem1}
Let $f(z)$ be an entire function of exponential type $\tau$ whose reduction in the real line, $f(x)$, be absolutely integrable in $[0,\infty), i.e.  \int_{0}^{\infty} |f(x)| \mbox{d}x < \infty$.
Let $J_{2l} (\lambda x)$ be the Bessel function of the first kind of order $2l$ for all non-negative integer $l$.
Let $f(z)$ has the expansion
\begin{equation}\label{eqn:theorem1function}
f(z)=z^{2m+1}\sum_{j=0}^{\infty} a_j z^{2j},
\end{equation} 
where $a_0\neq 0$, $m+l \geq 0$ for integer $m$. If $\lambda > \tau$, then
\begin{eqnarray}\label{eqn:theorem1Int}
\int_0^{\infty}f(x) J_{2l}(\lambda x)\mbox{d}x=0,\;\; l<m+1,
\end{eqnarray}
and
\begin{eqnarray}\label{eqn:theorem1IntSum}
\int_0^{\infty}\!\!\!\!\! f(x) J_{2l}(\lambda x)\mbox{d}x=\frac{1}{2}\!\sum_{j=0}^{l-m-1} a_j \frac{(l+m+j)!}{(l-m-j-1)!} \left(\frac{2}{\lambda}\right)^{2m+2j+2}, \; l\geq m+1 .
\end{eqnarray}
\end{theorem}
\begin{proof}
Let
\begin{align}\label{eqn:IntI1}
    I_1(\lambda) \, = \, \int_0^{\infty}f(x) J_{2l}(\lambda x)\mbox{d}x . 
\end{align}
We take the extension of this integral in the negative real line. To do this, we use the relationship between the Hankel function and the Bessel function provided by the connection formula \cite[p.75, 3.62.5]{watsontreatise},
\begin{eqnarray}\label{connection}
H_{\nu}^{(1)}(z e^{i \alpha \pi}) = e^{- i \alpha \nu \pi} H_{\nu}^{(1)}(z)-2 e^{- i \nu \pi} \frac{\sin (\alpha \nu\pi)}{\sin (\nu\pi)} J_{\nu}(z),
\end{eqnarray}
where $H_{\nu}^{(1)}(z)$ is a Hankel function of order $\nu$ and has a branch cut along the negative real axis $(-\infty,0]$ where it has a logarithmic singularity at $z=0$. We consider the case when $\nu=2l$ and let $\alpha=1$ in equation \eqref{connection} to get
\begin{eqnarray}\label{case1}
2 J_{2l}(z)= H^{(1)}_{2l}(z)-H^{(1)}_{2l}(z e^{i\pi}),
\end{eqnarray}
for all $l=0,1,2,\dots$. Applying this relationship into equation \eqref{eqn:IntI1} and note that in this case, $f(x)$ must be odd leads to the principal value integral
\begin{eqnarray} \label{eqn:theorem1PVI}
\int_0^{\infty}f(x) J_{2l}(\lambda x)\mbox{d}x =\frac{1}{2} PV  \int_{-\infty}^{\infty} f(x) H^{(1)}_{2l}(\lambda x)\, \mbox{d}x .
\end{eqnarray}

To determine the principal value, we consider the following contour integral
\begin{align}
     \oint_C f(z) H^{(1)}_{2l}(\lambda z)\, \mbox{d}z = 0,
\end{align}
where the path is shown in Figure \ref{fig:Figure1}, and the integral equals zero because no poles are enclosed within the contour. We choose the direction of the integration to be counterclockwise. Along the path-$AB$ and path-$CD$, we parametrize $z=xe^{i \alpha}$ where $ \alpha=0$ then
\begin{align}\label{eqn:pathAB}
\int_{AB} =  \int_{-L}^{-\epsilon} \,f(x) H^{(1)}_{2l}(\lambda x)\, \mbox{d}x, \quad \int_{CD} =  \int_{\epsilon}^{L} \,f(x) H^{(1)}_{2l}(\lambda x)\, \mbox{d}x.
\end{align}
Along the semicircle, the parametrization is $z=\epsilon e^{i \theta}$ so that we have
\begin{align}
    \int_{BC} = \, \int_{\pi}^0 \, f(\epsilon e^{i \theta }) \, H_{2l}^{(1)} (\lambda \epsilon e^{i \theta}) \,i \epsilon e^{i \theta}  \, \text{d} \theta.
\end{align}
For the path-$DE$ and path-$FA$, we have the following integrals,
\begin{align}\label{eqn:pathDEFA}
\int_{DE} = \int_{0}^{\gamma} \,f(L+iy) H^{(1)}_{2l}(\lambda (L+iy))\, i \mbox{d}y, \quad \int_{FA} = \int_{\gamma}^{0} \,f(-L+iy) H^{(1)}_{2l}(\lambda (-L+iy))\, i \mbox{d}y.
\end{align}
And for the path-$EF$, we have
\begin{align}\label{eqn:pathEF}
\int_{EF} = \int_{L}^{-L} f(x+i\gamma) H^{(1)}_{2l} (\lambda(x+i\gamma))\, \mbox{d}x.
\end{align}

Let us now estimate the integral along the path-$DE$. The Hankel function has the asymptotic expansion \cite[p.920,8.451.3, p.921,8.451.8]{gradstein2007zwillinger} 
\begin{align}\label{eqn:HankelAsymExp}
    H_\nu^{(1)} (z) = \, \sqrt{\frac{2}{\pi z}} \, e^{-i (z-\frac{\pi \, \nu}{2} - \frac{\pi}{4})} \, \left[ \sum_{k=0}^{n-1} \, \frac{(-1)^k}{(2 i z)^k k!} \frac{\Gamma(\nu+k+\frac{1}{2})}{\Gamma(\nu-k+\frac{1}{2})} \, + \, \Theta_1 \frac{(-1)^n}{(2 i z)^n k!} \frac{\Gamma(\nu+n+\frac{1}{2})}{\Gamma(\nu-n+\frac{1}{2})} \right], \,\, |z| \to \infty
\end{align}
for $Re(\nu)>-\frac{1}{2}$ and $|\text{Arg}(z)|<\pi$. For $Im(z)\geq 0$, the expansion has the upper bound \begin{align}\label{eqn:HankelBound}
    |H_\nu^{(1)} (z)| \leq \Bigg| \sqrt{\frac{2}{\pi z}} \, e^{- i z} \Bigg|.
\end{align}
Using this inequality in the upper half plane, we have the bound
\begin{eqnarray}
\left|\int_{DE} \right|&\leq & \int_0^{\gamma} |f(L+iy)| \,  |H^{(1)}_{2l} (\lambda (L+iy)) |\, \mbox{d}y \leq \sqrt{\frac{2}{\pi \lambda}} \int_0^{\gamma} |f(L+i y)| \, \mbox{d}y \label{eqn:DEestimate}.
\end{eqnarray}
Since $f(x)$ vanishes as $|x|\to \infty$, the integral $\int_0^{\gamma} |f(L+iy)| \mbox{d}y$ vanishes as $L\rightarrow\infty$ \cite{boas1954entire}. Then the integral along the path-$DE$ vanishes in the limit as $L\rightarrow\infty$.
The same result holds for the integral along the path-$FA$ by applying Lemma 3.2 in \cite{schmeisser2007approximation}. Estimating the integral along path-$EF$ as $L\rightarrow\infty$, we have the bound
\begin{eqnarray}
\left|\int_{EF}\right| \leq  \int_{-\infty}^{\infty} |f(x+i\gamma)| |H^{(1)}_{2l} (\lambda(x+i\gamma))|\, \mbox{d}x \leq  \sqrt{\frac{2}{\pi \lambda}}e^{-\lambda\gamma} \int_{-\infty}^{\infty} \frac{|f(x+i\gamma)|}{(x^2+\gamma^2)^{\frac{1}{4}}} \, \mbox{d}x . 
\end{eqnarray}

In the right hand side of the equation above, the integrand is an exponential type function thus we apply the inequality $\int_{-\infty}^{\infty} |g(x+iy)|^p \mbox{d}x \leq e^{p \tau |y| } \int_{-\infty}^{\infty}   |g(x)|^p \mbox{d}x$ where $g(z)$ is an exponential function of type $\tau$ \cite{boas1954entire}. For $p=1$, we then have
\begin{eqnarray}
\left|\int_{EF}\right| 
\leq  \, \sqrt{\frac{2}{\pi \lambda } }e^{-\gamma(\lambda-\tau)} \int_{-\infty}^{\infty} \frac{|f(x)|}{(x^2+\gamma^2)^{1/4}}\, \mbox{d}x
\leq M \sqrt{\frac{2}{\pi \lambda } }e^{-\gamma(\lambda-\tau)} .
\end{eqnarray}
When $\lambda > \tau$, the bound on the integral along $EF$ can be made arbitrarily small by translating $EF$ further corresponding to an arbitrarily large $\gamma'>\gamma$. This implies that the integral along $EF$ vanishes when $\lambda > \tau$. The vanishing of the integral depended on the the fact that $f(z)$ is an exponential function of type $\tau$.

For $\lambda > \tau$, we now have the following
\begin{align}
    0 \, = \, \oint_C f(z) H^{(1)}_{2l}(\lambda z)\, \mbox{d}z = & \, \int_{-L}^{-\epsilon} \,f(x) H^{(1)}_{2l}(\lambda x)\, \mbox{d}x + \, \int_{\pi}^0 \, f(\epsilon e^{i \theta }) \, H_{2l}^{(1)} (\lambda \epsilon e^{i \theta}) \,i \epsilon e^{i \theta}  \, \text{d} \theta + \int_{\epsilon}^{L} \,f(x) H^{(1)}_{2l}(\lambda x)\, \mbox{d}x.
\end{align}
Rearranging the terms above and taking the limit as $\epsilon\to 0$ and $L\to \infty$, we write
\begin{align}
     \label{eqn:epsilonPV} \lim_{\substack{\epsilon\to 0 \\ L\to \infty}} \left[ \int_{-L}^{-\epsilon} \, f(x) H^{(1)}_{2l}(\lambda x)\, \mbox{d}x + \int_{\epsilon}^{L} \,f(x) H^{(1)}_{2l}(\lambda x)\, \mbox{d}x \right]= - \lim_{\epsilon\to 0} \, \int_{\pi}^0 \, f(\epsilon e^{i \theta }) \, H_{2l}^{(1)} (\lambda \epsilon e^{i \theta}) \,i \epsilon e^{i \theta}  \, \text{d} \theta .
\end{align}
The left-hand side of equation \eqref{eqn:epsilonPV} is now the principal value integral in equation \eqref{eqn:theorem1PVI}. 
The value of this integral must now come from the integral around the small semicircle. That is,
\begin{eqnarray}\label{integralnew}
PV \int_{-\infty}^{\infty} \,f(x) H^{(1)}_{2l}(\lambda x)\, \mbox{d}x  = - \, \lim_{\epsilon\to 0} \, \int_{\pi}^0 \, f(\epsilon e^{i \theta }) \, H_{2l}^{(1)} (\lambda \epsilon e^{i \theta}) \,i \epsilon e^{i \theta}  \, \text{d} \theta .
\end{eqnarray}
To evaluate the integral on the right hand side of the equation \eqref{integralnew}, we write $H_{2l}^{(1)}(z)$ in such a way that its singularity at $z=0$ is manifested. We make use of the relationship $H_{\nu}^{(1)}(z)=J_{\nu}(z) + i Y_{\nu}(z)$, where $Y_{\nu}(z)$ is a Neumann function. When $\nu=n$, an integer, we have \cite[p.911, 8.403.2]{gradstein2007zwillinger}
\begin{align}
Y_n(z) = & \frac{2}{\pi} J_{n}(z) \, \ln \left(\frac{z}{2} \right) - \frac{1}{\pi} \,  \sum_{s=0}^{n-1} \frac{(n-s-1)!}{s!} \left(\frac{z}{2}\right)^{2s-n} -\frac{1}{\pi} \sum_{s=0}^{\infty} (-1)^s \frac{1}{s!(s+n)!} \left(\frac{z}{2}\right)^{n+2s} \left[\psi(s+1)+\psi(s+n+1)\right].
\end{align}
Then for every positive integer $n$,
\begin{eqnarray}\label{eqn:hankelsum}
H^{(1)}_{n}(z)=-\frac{i}{\pi}\sum_{s=0}^{n-1} \frac{(n-s-1)!}{s!}\left(\frac{z}{2}\right)^{2s-n} + \phi(z),
\end{eqnarray}
where $\phi(z)$ carries the logarithmic singularity of $H_{n}^{(1)}(z)$.
Now for this case, the function $f(x)$ must be odd and is entire, we can write $f(x)$ as in equation \eqref{eqn:theorem1function} where $a_0\neq 0$ and $m$ is a natural number. Using equations \eqref{eqn:theorem1function} and \eqref{eqn:hankelsum} for $n=2l$ and $z= \epsilon e^{i \theta }$, and substitute back into equation \eqref{integralnew} results to
\begin{eqnarray}\label{eqn:PVint1}
PV \int_{-\infty}^{\infty} \,f(x) H^{(1)}_{2l}(\lambda x)\, \mbox{d}x = -\frac{1}{\pi} \,  \sum_{j=0}^{\infty} \, \sum_{s=0}^{2l-1}a_j\frac{(2l-s-1)!}{s!} \left(\frac{\lambda}{2}\right)^{2s-2l} \, \lim_{\epsilon\rightarrow 0} \, \epsilon^{2j+2m+2-2l+2s} \,  \int_{\pi}^0 e^{2i\theta(m +1 -l +j+s)}\, \mbox{d}\theta,
\end{eqnarray}
where the term involving $\phi(z)$ has been dropped because it vanishes as $\epsilon\rightarrow 0$. On the right-hand side of equation \eqref{eqn:PVint1}, the integral is equal to zero when $m +1 -l +j+s \neq 0$ but it is nonzero when $m +1 -l +j+s = 0$, so that
\begin{align}
\label{eqn:ExpInt}
    \int_{\pi}^0 e^{2i\theta(m +1 -l +j+s)}\, \mbox{d}\theta = - \pi \delta_{m +1 -l +j+s,0}.
\end{align}

Therefore, the integral on the right-hand side of equation \eqref{eqn:PVint1} is zero if $l<m+1$ for all $j=0,1,2,\dots$ and $s=0,1,\dots, 2l-1$. Consequently, the principal value in \eqref{eqn:theorem1PVI} is equal to zero, which proves equation \eqref{eqn:theorem1Int}.

On the other hand, the integral on the right-hand side of equation \eqref{eqn:PVint1} becomes nonzero when $l \geq m+1$ so that applying the result \eqref{eqn:ExpInt} into equation \eqref{eqn:PVint1}, we are then left with
\begin{eqnarray}
PV \int_{-\infty}^{\infty} \,f(x) H^{(1)}_{2l}(\lambda x)\, \mbox{d}x = \,  \sum_{j=0}^{\infty} \, a_j \frac{(l+m+j)!}{(l-m-1-j)!}  \left(\frac{\lambda}{2}\right)^{-2m-2j-2}.
\end{eqnarray}   
Simplifying the equation above, we see that the terms becomes zero for values of
$j=l-m, l-m+1, l-m+2,\dots$ thus the series terminates at $j=l-m-1$ so that the desired principal value is given by
\begin{eqnarray}
PV \int_{-\infty}^{\infty} \,f(x) H^{(1)}_{2l}(\lambda x)\, \mbox{d}x = \, \sum_{j=0}^{l-m-1} \, a_j \, \frac{(l+m+j)!}{(l-m-1-j)!} \left(\frac{2}{\lambda}\right)^{2m+2j+2}, \quad l \geq m+1. 
\end{eqnarray}  
We substitute this result back into equation \eqref{eqn:theorem1PVI}, thereby proving equation \eqref{eqn:theorem1IntSum} and completing the proof of Theorem 1.
\end{proof}

\textit{Remark:} When $\lambda < \tau$, the contribution from the path-$EF$ does not vanish and exponentially small terms are obtained from this path.

\subsection{Case 2}

\begin{theorem}\label{theorem3}
Let $h(z)$ be an entire function of exponential type $\tau$ whose reduction in the real line, $h(x)$, be absolutely integrable in $[0,\infty), i.e.  \int_{0}^{\infty} |h(x)| \mbox{d}x < \infty$. Let $J_{2l+1} (\lambda x)$ be the Bessel function of the first kind of order $2l+1$ for all non-negative integer $l$. Let $h(z)$ has the expansion
\begin{equation}\label{eqn:Theorem3ExpFn}
h(z)=z^{2m}\sum_{j=0}^{\infty} b_j \, z^{2j}
\end{equation}
where $b_0\neq 0$, $m+l \geq 0$ for integer $m$. If $\lambda > \tau$, then
\begin{eqnarray} \label{eqn:Theorem2Int}
\int_0^{\infty}h(x) J_{2l+1}(\lambda x)\mbox{d}x=0,\;\; l<m ,
\end{eqnarray}
and
\begin{eqnarray} \label{eqn:theorem2IntSum}
\int_0^{\infty}h(x) J_{2l+1}(\lambda x)\mbox{d}x=\frac{1}{2}\sum_{j=0}^{l-m} b_j \frac{(l+m+j)!}{(l-m-j)!} \left(\frac{2}{\lambda}\right)^{2m+2j+1}, \;\; l\geq m .
\end{eqnarray}
\end{theorem}
\begin{proof}
The proof is identical to that of Theorem 1, but we now consider the case when $\nu$ is odd so that we have the integral
\begin{align}
    I_2 (\lambda) = \int_0^{\infty} \, h(x) \, J_{2l+1}(\lambda x) \, \text{d}x.
\end{align}
Using the connection formula in equation \eqref{connection} for $\alpha=1$ and $\nu=2l+1$ and note that $h(x)$ must be an even function, we can then write the integral $I_2$ using its principal value,
\begin{align} \label{eqn:theorem2PVI}
    \int_0^{\infty} \, h(x) \, J_{2l+1}(\lambda x) \, \text{d}x = \frac{1}{2} PV  \int_{-\infty}^{\infty} \, h(x) H_{2l+1}^{(1)}(\lambda x) \, \text{d}x.
\end{align}

To evaluate the right-hand side of the equation above, we perform an integration for the contour integral \begin{align}
     \oint_C h(z) H^{(1)}_{2l+1}(\lambda z)\, \mbox{d}z = 0,
\end{align}
along the same path shown in Figure \ref{fig:Figure1}. We note that there are no poles in the upper half plane, so the integral is equal to zero. The same parametrizations from Theorem 1 are also done for each path, following the same direction. Using the same estimate that we got in equation \eqref{eqn:DEestimate}, the contributions along the path-$DE$ and path-$FA$ vanish in the limit as $L\to \pm \infty$ respectively \cite{boas1954entire,schmeisser2007approximation}. The same estimate is also done for the contribution along the path-$EF$ which is equal to zero provided that $\lambda$ is arbitrarily large than the type $\tau$. 

Then for $\lambda > \tau$ and applying the limit as $\epsilon\to 0$ and $L\to \infty$, we then have
\begin{align}
     \label{eqn:epsilonPV2} \lim_{\substack{\epsilon\to 0 \\ L\to \infty}} \left[ \int_{-L}^{-\epsilon} \, h(x) H^{(1)}_{2l+1}(\lambda x)\, \mbox{d}x + \int_{\epsilon}^{L} \, h(x) H^{(1)}_{2l+1}(\lambda x)\, \mbox{d}x \right]= - \lim_{\epsilon\to 0} \, \int_{\pi}^0 \, h(\epsilon e^{i \theta }) \, H_{2l+1}^{(1)} (\lambda \epsilon e^{i \theta}) \,i \epsilon e^{i \theta}  \, \text{d} \theta. 
\end{align}
The left-hand side of equation \eqref{eqn:epsilonPV2} now corresponds to the principal value integral in equation \eqref{eqn:theorem2PVI}.
The remaining contribution is then obtained from the semicircle which is given by
\begin{align}\label{integralnew2}
PV \int_{-\infty}^{\infty} \, h(x) H^{(1)}_{2l+1}(\lambda x)\, \mbox{d}x  = - \, \lim_{\epsilon\to 0} \, \int_{\pi}^0 \, h(\epsilon e^{i \theta }) \, H_{2l+1}^{(1)} (\lambda \epsilon e^{i \theta}) \, i \epsilon e^{i \theta}  \, \text{d} \theta .
\end{align}
We now write the Hankel function $H_\nu^{(1)} (z)$ using equation \eqref{eqn:hankelsum} for $\nu=2l+1$ and $h(x)$ as in equation \eqref{eqn:Theorem3ExpFn}, then
\begin{align}\label{eqn:PVint2}
    PV \int_{-\infty}^{\infty} \, h(x) H^{(1)}_{2l+1}(\lambda x)\, \mbox{d}x = -\frac{1}{\pi} \, \sum_{j=0}^\infty \, \sum_{k=0}^{2l} \, b_j  \frac{\,  (2l-k)!}{k!} \left( \frac{\lambda}{2} \right)^{2k-2l-1} \, \lim_{\epsilon\to 0} \, \epsilon^{2j+2m-2l+2k} \, \int_{\pi}^0 \, e^{2 i \theta (j+m-l+k)} \, \text{d}\theta.
\end{align}
Here, the integral on the right hand side of the equation above becomes zero when $m-l+j+k \neq 0$ but it is nonzero when $m-l+j+k=0$ so that
\begin{align}\label{eqn:ExpInt2}
    \int_{\pi}^0 \, e^{2 i \theta (j+m-l+k)} \, \text{d}\theta = \, - \pi \, \delta_{m-l+j+k,0}.
\end{align}

For all $j=0,1,2,\dots$ and $k=0,1,\dots, 2l$, the integral on the right-hand side of \eqref{eqn:PVint2} is zero when $l<m$, thus the principal value in equation \eqref{eqn:theorem2PVI} becomes zero. With this, equation \eqref{eqn:Theorem2Int} is proven.

The integral on the right-hand side of \eqref{eqn:PVint2} then becomes nonzero when $l \geq m$. Substituting the result \eqref{eqn:ExpInt2} into equation \eqref{eqn:PVint2}, we now get the following
\begin{align}
    PV \int_{-\infty}^{\infty} \, h(x) H^{(1)}_{2l+1}(\lambda x)\, \mbox{d}x = \, \sum_{j=0}^\infty \, b_j \, \frac{(l+m+j)! \, }{(l-m-j)!} \, \left( \frac{\lambda}{2} \right)^{-2m-2j-1}.
\end{align}
It can be seen here that for $j=l-m+1,l-m+2,l-m+3,\dots$, the value of these terms become zero such that the series terminates at $j=l-m$. We then arrive with
\begin{align}
    PV \int_{-\infty}^{\infty} \, h(x) H^{(1)}_{2l+1}(\lambda x)\, \mbox{d}x = \, \sum_{j=0}^{l-m} \, b_j \, \frac{(l+m+j)! \, }{(l-m-j)!} \, \left( \frac{2}{\lambda} \right)^{2m+2j+1}, \quad l \geq m.
\end{align}
By using the result above to equation \eqref{eqn:theorem2PVI}, then we have proven equation \eqref{eqn:theorem2IntSum} and the proof of Theorem 2 is complete.
\end{proof}

\section{Hankel transform of the reciprocal of a product of gamma functions}

Let us explore examples of Hankel transforms. We consider the reciprocal of the product of two gamma functions which has the form $ 1 / (\Gamma (\alpha + \beta z) \Gamma (\alpha - \beta z)) $. This is a special case of the general result obtained in \cite{napalkov2006entire}.
We first show that this is an exponential type function.

\begin{lemma}
For sufficiently large $|z|$, there exists a number $C>0$ such that
\begin{equation}\label{eqn:Gammalemma}
\frac{1}{|\Gamma (\alpha + \beta z) \Gamma (\alpha - \beta z)|} \leq \frac{C e^{2 \alpha}}{2 \pi \, |\beta z|^{2 \alpha -1}} \, e^{\beta \pi \, \text{Im}(z)}, \quad |z| \to \infty,
\end{equation}
for $\alpha,\beta>0$ in the sector $|\text{Arg} \, z| <\pi$. Then,  $ 1 / (\Gamma (\alpha + \beta z) \Gamma (\alpha - \beta z)) $ is an exponential function of type $\beta \pi$.
\end{lemma}

\begin{proof}
To show that this function is an exponential type function, we start with the asymptotic expansion of the gamma function given by $\Gamma(z) = e^{-z} \, z^z \, (2 \pi /z)^{1/2} (1+ o(1))$ as $|z| \to \infty$ in the sector $|\text{Arg} \, z| \leq (\pi - \delta) (<\pi)$ \cite[p. 140, 5.11.3]{olver2010nist}, then we have
\begin{equation}\label{eqn:Gammareciprocal}
    \frac{1}{\Gamma(\alpha + \beta z)} = \, \frac{1}{\sqrt{2 \pi}} \, \text{exp} \left[ (\alpha + \beta z) - \left(\alpha + \beta z - \frac{1}{2} \right) \, \ln{(\alpha + \beta z)} + o(1) \right]
\end{equation}
for any number $\alpha$ and $\beta$ for $|z| \to \infty$ in the sector $|\text{Arg} \, (\alpha + \beta z)| < \pi$. For positive $\alpha$ and $\beta$, the reciprocal of the product of the gamma functions holds in the upper and lower half-planes for sufficiently large $z$.

By expressing the logarithmic function using its principal value for a complex number, the following estimate can be established
\begin{align}
    \left(\alpha + \beta z - \frac{1}{2} \right) \, \ln{(\alpha + \beta z)} 
    = \, \ln{|\beta z|^{\alpha - \frac{1}{2}}} + \beta z \, \ln{|\beta z|} + i \, \beta z \, \text{Arg} \, (\alpha + \beta z) + \bigO(1),
\end{align}
as $|z| \to \infty$, which is to be substituted back into equation \eqref{eqn:Gammareciprocal} so we have now
\begin{align}
    \frac{1}{\Gamma(\alpha + \beta z)} = \, \frac{1}{\sqrt{2 \pi}} \, \text{exp} \left[ (\alpha + \beta z) - \ln{|\beta z|^{\alpha - \frac{1}{2}}} + \beta z \, \ln{|\beta z|} + i \, \beta z \, \text{Arg} \, (\alpha + \beta z) + \bigO(1) \right].
\end{align}
The resulting estimate holds as well for substituting $z$ with $-z$. 
Hence, we have the asymptotic behavior
\begin{align}
    \frac{1}{|\Gamma (\alpha + \beta z) \Gamma (\alpha - \beta z)|} & = \, \frac{e^{2 \alpha}}{2\pi \, |\beta z|^{2 \alpha -1}} \, \text{exp} \left[ i \, \beta z \, \{ \text{Arg} \, (\alpha + \beta z) - \text{Arg} \, (\alpha - \beta z) \} + \bigO(1) \right]. 
    \end{align}
Further, the difference of $\text{Arg} \, (\alpha + \beta z) - \text{Arg} \, (\alpha - \beta z) $ is equal to $ -\pi$ so that we arrive with
\begin{align} 
\frac{1}{|\Gamma (\alpha + \beta z) \Gamma (\alpha - \beta z)|} = \, \frac{e^{2 \alpha}}{2\pi \, |\beta z|^{2 \alpha -1}} \, \text{exp} \left[ -i \, \beta \pi z \, + \bigO(1) \right], \quad |z| \to \infty .
\end{align}

Then for sufficiently large $|z|$, there exists a number $C>0$ such that the inequality in equation \eqref{eqn:Gammalemma} holds. This establishes that the reciprocal of the product of the gamma functions is an exponential function of type $\beta \pi$.
\end{proof}

Let us now apply the established theorem from the previous section to the exponential type function in Lemma 1.

\begin{proposition}
For $m=0,1,2,\dots$ and non-negative $l\geq m+1$, we have
\begin{align}\label{eqn:theoremGamma1}
    \int_0^\infty \, \frac{x^{2m+1} \, J_{2l} (\lambda x)}{\prod_{k=1}^M \, \Gamma(\alpha_k + \beta_k x) \, \Gamma(\alpha_k - \beta_k x)} \, & \, \text{d}x \, = \, \frac{1}{2} \, \sum_{s=0}^{l-m-1} \, \frac{a_s^{(M)}}{(2s)!} \, \frac{(l+m+s)!}{(l-m-s-1)!} \, \left( \frac{2}{\lambda} \right)^{2(m+s+1)},
\end{align}
where $\alpha_k,\beta_k>0$ for $k=1,2,\dots, M<\infty$, and
\begin{align}\label{eqn:GammaCoeff}
    a_s^{(M)} = \frac{d^{2s}}{dx^{2s}} \, \frac{1}{\prod_{k=1}^M \, \Gamma(\alpha_k + \beta_k \, x) \, \Gamma(\alpha_k - \beta_k \, x)} \bigg|_{x=0},
\end{align}
for $s=0,1,2,\dots$, and the following restrictions hold
\begin{align} \label{eqn:Alpharestriction1}
    \sum_{k=1}^M \, \alpha_k > m+\frac{1}{2} \, & \left( M + \frac{3}{2}  \right), \\ \label{eqn:Lambdarestriction}
    \lambda \, \geq \, \pi \, \sum_{k=1}^M \, \beta_k. &
\end{align}
Moreover, under the same restrictions \eqref{eqn:Alpharestriction1}, \eqref{eqn:Lambdarestriction} and for $l<m+1$, we have
\begin{align}
    \label{eqn:theoremGamma1zero}
    \int_0^\infty \, \frac{x^{2m+1} \, J_{2l} (\lambda x)}{\prod_{k=1}^M \, \Gamma(\alpha_k + \beta_k x) \, \Gamma(\alpha_k - \beta_k x)} \, & \, \text{d}x \, = \, 0.
\end{align}
\end{proposition}

\begin{proof}
We apply Theorem 1 where $f(x)$ is given by
\begin{align}\label{eqn:GammaFnOdd}
    f(x) = \frac{x^{2m+1}}{\prod_{k=1}^M \, \Gamma(\alpha_k + \beta_k x) \, \Gamma(\alpha_k - \beta_k x)}.
\end{align}
Applying Lemma 1 for $f(x)$, we see that this is an exponential function of type $\pi \, \sum_{k=1}^M \, \beta_k$. The series form of $f(x)$ can be obtained by taking the Taylor series about $x=0$ of the denominator which reveals that only even powers contribute to its series expansion so that
\begin{align}\label{eqn:TwoGammaSeries}
    \frac{1}{\prod_{k=1}^M \, \Gamma(\alpha_k + \beta_k x) \, \Gamma(\alpha_k - \beta_k x)} = \sum_{s=0}^\infty
    \, \frac{1}{(2s)!} \, a_s^{(M)} \, x^{2s},
\end{align}
where $a_s^{(M)}$ is precisely given by \eqref{eqn:GammaCoeff}.

The restrictions in \eqref{eqn:Alpharestriction1} and \eqref{eqn:Lambdarestriction} are obtained by performing the contour integration for the integral in equation \eqref{eqn:theoremGamma1}, following the contour path in Figure 1. Estimating the integral along the path-$DE$ using the inequality in equation \eqref{eqn:HankelBound} and Lemma 1, the upper bound of the integral can be written as
\begin{align} \label{eqn:IntGammaDEbound}
    \Bigg| \int_0^{\gamma} \, \frac{(L+iy)^{2m+1} \, H_{2l}^{(1)} (\lambda (L+iy))}{\prod_{k=1}^M \, \Gamma(\alpha_k + \beta_k (L+iy)) \, \Gamma(\alpha_k - \beta_k (L+iy))} \, i \text{d}y \Bigg| \, & \leq \, \left( \prod_{k=1}^M \, \frac{C e^{2 \lambda_k}}{2 \pi \beta_k^{2 \alpha_k -1}} \right) \, \sqrt{\frac{2}{\pi \lambda}}  \nonumber \\
    & \quad \times \, \,  e^{-\gamma (\lambda -  \pi \sum_{k=1}^M \beta_k} )\, \int_0^\gamma \frac{1}{|L+iy|^{2 \sum_{k=1}^m \alpha_k -2m - \frac{3}{2}} } \text{d}y.
\end{align}
In the last line of the equation above, this becomes zero for sufficiently large $\gamma$ provided that the restrictions  \eqref{eqn:Alpharestriction1} and \eqref{eqn:Lambdarestriction} holds, in accordance to Theorem 1.

Thus, it follows from Theorem 1 that equation \eqref{eqn:theoremGamma1} holds true for $l\geq m+1$ as well as equation \eqref{eqn:theoremGamma1zero} for $l < m+1$. The proposition is proved.
\end{proof}

\begin{proposition}
For $m=0,1,2,\dots$ and non-negative $l\geq m$, we have
\begin{align}\label{eqn:theoremGamma2}
    \int_0^\infty \, \frac{x^{2m} \, J_{2l+1} (\lambda x)}{\prod_{k=1}^M \, \Gamma(\alpha_k + \beta_k x) \, \Gamma(\alpha_k - \beta_k x)} \, & \, \text{d}x \, = \, \frac{1}{2} \, \sum_{s=0}^{l-m} \, \frac{a_s^{(M)}}{(2s)!} \, \frac{(l+m+s)!}{(l-m-s)!} \, \left( \frac{2}{\lambda} \right)^{2m+2s+1},
\end{align}
where and $\alpha_k,\beta_k>0$ for $k=1,2,\dots, M<\infty$ and $a_s^{(M)}$ is given by \eqref{eqn:GammaCoeff}, and the restrictions 
\begin{align} \label{eqn:Alpharestriction2}
    \sum_{k=1}^M \, \alpha_k > m+\frac{1}{2} \, & \left( M + \frac{1}{2}  \right), \\
    \label{eqn:Lambdarestriction1}
    \lambda \, \geq \, \pi \, \sum_{k=1}^M \, \beta_k. &
\end{align}
are satisfied.
Moreover, under the same restrictions \eqref{eqn:Alpharestriction2}, \eqref{eqn:Lambdarestriction1} and for $l<m$, we have
\begin{align}
    \label{eqn:theoremGamma2zero}
   \int_0^\infty \, \frac{x^{2m} \, J_{2l+1} (\lambda x)}{\prod_{k=1}^M \, \Gamma(\alpha_k + \beta_k x) \, \Gamma(\alpha_k - \beta_k x)} \, & \, \text{d}x \, = \, 0.
\end{align}
\end{proposition}

\begin{proof}
We also apply Theorem 2 where $h(x)$ is given by
\begin{align}\label{eqn:GammaFnEven}
    h(x) \, = \, \frac{x^{2m}}{\prod_{k=1}^M \, \Gamma(\alpha_k + \beta_k x) \, \Gamma(\alpha_k - \beta_k x)}.
\end{align}
By applying Lemma 1, $h(x)$ is now an exponential function of type $\pi \, \sum_{k=1}^M \, \beta_k$. Its series form is obtained by using equation \eqref{eqn:TwoGammaSeries}.

The restrictions \eqref{eqn:Alpharestriction2} and \eqref{eqn:Lambdarestriction1} can be shown by performing the contour integration for the integral in \eqref{eqn:theoremGamma2} using the contour in Figure 1. The same estimate from Theorem 2 is also done for the path along $DE$ where it has the upper bound
\begin{align}
    \Bigg| \int_0^{\gamma} \, \frac{(L+iy)^{2m} \, H_{2l+1}^{(1)} (\lambda (L+iy))}{\prod_{k=1}^M \, \Gamma(\alpha_k + \beta_k (L+iy)) \, \Gamma(\alpha_k - \beta_k (L+iy))} \, i \text{d}y \Bigg| \, & \leq \, \left( \prod_{k=1}^M \, \frac{C e^{2 \lambda_k}}{2 \pi \beta_k^{2 \alpha_k -1}} \right) \, \sqrt{\frac{2}{\pi \lambda}}  \nonumber \\
    & \quad  \times \, \,  e^{-\gamma (\lambda -  \pi \sum_{k=1}^M \beta_k} )\, \int_0^\gamma \frac{1}{|L+iy|^{2 \sum_{k=1}^m \alpha_k -2m - \frac{1}{2}} } \text{d}y.
\end{align}
In the right hand side of the equation above, it can be seen on the last line this becomes zero for sufficiently large $\gamma$ provided that the conditions in \eqref{eqn:Alpharestriction2} and \eqref{eqn:Lambdarestriction1} are satisfied.

It also follows from Theorem 2 that equation \eqref{eqn:theoremGamma2} holds for $l<m$ as well as equation \eqref{eqn:theoremGamma2zero} for $l \geq m$. Thus, the proposition is proved.
\end{proof}

The coefficients $a_s^{(M)}$'s, except for $a_0^{(M)}$, are in terms of the derivatives of the digamma function $\psi(z)=\Gamma'(z) / \Gamma(z)$. Direct differentiation and substitution give the following first few coefficients:
\begin{align}
a_0^{(M)} \, = & \, \frac{1}{\prod_{k=1}^M \, \left( \Gamma(\alpha_k) \right)^2}, \quad M=1,2,3,\dots, \\
a_1^{(M)} \, = & \, -\frac{2}{\prod_{k=1}^M \, \left( \Gamma(\alpha_k) \right)^2} \sum_{k=1}^M \, \beta_k^2 \, \psi^{(1)}(\alpha_k), \quad M=1,2,3,\dots, \\
a_2^{(1)} \, = & \, \frac{\beta_1^4}{\left(\Gamma(\alpha_1) \right)^2} \, \left[ 12 \left( \psi^{(1)}(\alpha_1) \right)^2 - 2 \psi^{(3)}(\alpha_1) \right], \\
a_2^{(2)} \, = & \, \frac{\beta_1^2 \, \beta_2^2}{ \left(\Gamma(\alpha_1) \right)^2 \, \left( \Gamma(\alpha_2) \right)^2} \, 24 \, \left(\psi^{(1)}(\alpha_1)\right)^2 \, \left(\psi^{(1)}(\alpha_2)\right)^2 \, + \, \frac{\beta_1^4}{ \left(\Gamma(\alpha_1) \right)^2 \, \left( \Gamma(\alpha_2) \right)^2} \, \left[ 12 \left( \psi^{(1)}(\alpha_1) \right)^2 - 2 \psi^{(3)}(\alpha_1) \right] \nonumber \\
& \qquad + \, \frac{\beta_2^4}{ \left(\Gamma(\alpha_1) \right)^2 \, \left( \Gamma(\alpha_2) \right)^2} \, \left[ 12 \left( \psi^{(1)}(\alpha_2) \right)^2 - 2 \psi^{(3)}(\alpha_2) \right], \\
a_2^{(M)} \, = & \, \frac{1}{\prod_{k=1}^M \, \left(\Gamma(\alpha_k) \right)^2} \, \sum_{l=1}^M \, \beta_l^4 \, \left[ 12 \left( \psi^{(1)}(\alpha_l) \right)^2 - 2 \psi^{(3)}(\alpha_l) \right] \nonumber \\
& \quad + \, \frac{24}{\prod_{k=1}^M \, \left(\Gamma(\alpha_k) \right)^2} \, \sum_{l=1}^{M-1} \, \beta_l^2 \, \psi^{(1)}(\alpha_l) \, \sum_{l'>l}^{M} \, \beta_{l'}^2 \, \psi^{(1)}(\alpha_{l'}), \quad M=3,4,\dots,
\end{align}
where $\psi^{(n)}(z)$ is the $n$-th derivative of $\psi(z)$, also known as the $n$-th polygamma function.

Using these coefficients, special cases of integrals \eqref{eqn:theoremGamma1} and \eqref{eqn:theoremGamma2} can be explicitly evaluated. For example, under the restrictions given by equations \eqref{eqn:Alpharestriction1} and \eqref{eqn:Lambdarestriction}
in equation \eqref{eqn:theoremGamma1} for $l=m+1$ and $l=m+2$, we have the integrals respectively
\begin{align} \label{eqn:theorem4n1}
   \int_0^\infty \, \frac{x^{2m+1} \, J_{2m+2} (\lambda x)}{\prod_{k=1}^M \, \Gamma(\alpha_k + \beta_k x) \, \Gamma(\alpha_k - \beta_k x)} \, \text{d}x \, = & \, \frac{(2m+1)!}{2 \, \prod_{k=1}^M \, \left( \Gamma(\alpha_k) \right)^2} \, \left( \frac{2}{\lambda} \right)^{2(m+1)}, \\ \label{eqn:theorem4n2}
   \int_0^\infty \, \frac{x^{2m+1} \, J_{2m+4} (\lambda x)}{\prod_{k=1}^M \, \Gamma(\alpha_k + \beta_k x) \, \Gamma(\alpha_k - \beta_k x)} \, \text{d}x \, = & \, \frac{(2m+2)!}{2 \, \prod_{k=1}^M \, \left( \Gamma(\alpha_k) \right)^2} \, \left( \frac{2}{\lambda} \right)^{2(m+1)} \left[ 1 \, - \, \frac{4 (2m+3)}{\lambda^2} \, \sum_{k=1}^M \, \beta_k^2 \, \psi^{(1)} (\alpha_k) \right],
\end{align}
for all $m=0,1,2,\dots$. Also under the restrictions given by equations \eqref{eqn:Lambdarestriction} and \eqref{eqn:Alpharestriction2} in equation \eqref{eqn:theoremGamma2} for $l=m$ and $l=m+1$, we have the integrals
\begin{align} \label{eqn:theorem5r0}
   \int_0^\infty \, \frac{x^{2m} \, J_{2m+1} (\lambda x)}{\prod_{k=1}^M \, \Gamma(\alpha_k + \beta_k x) \, \Gamma(\alpha_k - \beta_k x)} \, \text{d}x \, = & \, \frac{(2m)!}{2 \, \prod_{k=1}^M \, \left( \Gamma(\alpha_k) \right)^2} \, \left( \frac{2}{\lambda} \right)^{2m+1}, \\ \label{eqn:theorem5r1}
   \int_0^\infty \, \frac{x^{2m} \, J_{2m+3} (\lambda x)}{\prod_{k=1}^M \, \Gamma(\alpha_k + \beta_k x) \, \Gamma(\alpha_k - \beta_k x)} \, \text{d}x \, = & \, \frac{(2m+1)!}{2 \, \prod_{k=1}^M \, \left( \Gamma(\alpha_k) \right)^2} \, \left( \frac{2}{\lambda} \right)^{2m+1}  \left[ 1 \, - \, \frac{8 (m+1)}{\lambda^2} \, \sum_{k=1}^M \, \beta_k^2 \, \psi^{(1)} (\alpha_k) \right],
\end{align}
for all $m=0,1,2,\dots$. We also have the following special cases for $m=0$ and $\beta_k=1$ and $\alpha_k=1,2$ respectively,
\begin{align}\label{eqn:GammaSpeCase}
    \int_0^\infty \, \frac{J_1 \, (\pi \, x)}{\Gamma(1+x) \, \Gamma(1-x)} \, \text{d}x \, = \, \, \frac{1}{\pi} \, , \qquad    \int_0^\infty \, \frac{x \, J_2 \, (\pi \, x)}{\Gamma(2+x) \, \Gamma(2-x)} \, \text{d}x \, = \, \, \frac{2}{\pi^2}.
\end{align}

Further, we can obtain identities for the value of $1/\pi$ by considering $M=1$ and $\beta_1=1, \alpha_1= m+n+1/2$ in equations \eqref{eqn:theorem5r0} and \eqref{eqn:theorem4n1} respectively, we have the following 
\begin{align}
    \frac{1}{\pi} =  & \frac{\left[ \prod_{k=1}^{m+n} (2k-1) \right]^2 \, \lambda^{2m+1}}{(2m)! \, 2^{4m+2n}} \int_0^{\infty} \frac{x^{2m+1} \, J_{2m+1} (\lambda x)}{\Gamma(m+n+\frac{1}{2}+x) \Gamma(m+n+\frac{1}{2}-x)}, \\
    \frac{1}{\pi} = &  \frac{\left[ \prod_{k=1}^{m+n} (2k-1) \right]^2 \, \lambda^{2m+2}}{(2m+1)! \, 2^{4m+2n+1}} \int_0^{\infty} \frac{x^{2m+1} \, J_{2m+2} (\lambda x)}{\Gamma(m+n+\frac{1}{2}+x) \Gamma(m+n+\frac{1}{2}-x)},
\end{align}
for $m=0,1,2,\dots$, $n=1,2,3,\dots$, and $\lambda > \pi$.

\end{document}